\documentclass{amsart}
\usepackage[pdftex]{graphicx}
\usepackage{amsmath}
\usepackage{amssymb}
\usepackage{amsfonts}
\usepackage{color}

\newtheorem{theor}{Theorem}

\newtheorem{lm}[theor]{Lemma}

\begin{document}

\title{Routh's Theorem for Tetrahedra}
\author{Semyon Litvinov}
\address{Penn State Hazleton, 76 University Drive, Hazleton, PA 18202, USA}
\email{snl2@psu.edu}
\author{Franti\v sek Marko}
\address{Penn State Hazleton, 76 University Drive, Hazleton, PA 18202, USA}
\email{fxm13@psu.edu}

\begin{abstract}
We give a geometric proof of the Routh's theorem for tetrahedra. 
\end{abstract}
\keywords{}

\maketitle
\section*{Introduction}

The purpose of this article is to give a geometrical generalization of the following classical Routh's theorem, see p.82 of \cite{r}.

\begin{theor}
Let $ABC$ be an arbitrary triangle of area $1$, a point $K$ lie on the line segment $BC$, 
a point $L$ lie on the line segment $AC$ and a point $M$ lie on the line segment $AB$ 
such that the ratio $\frac{|AM|}{|MB|}=x$, $\frac{|BK|}{|KC|}=y$ and $\frac{|CL|}{|LA|}=z$. 
Denote by $P$ the point of intersection of lines $AK$ and $CM$, by $Q$ the point of intersection 
of lines $BL$ and $AK$, and by $R$ the point of intersection of lines $CM$ and $BL$ - see Figure \ref{fig1}.
\begin{figure}[ht]\centering
\includegraphics[height=1.5in]{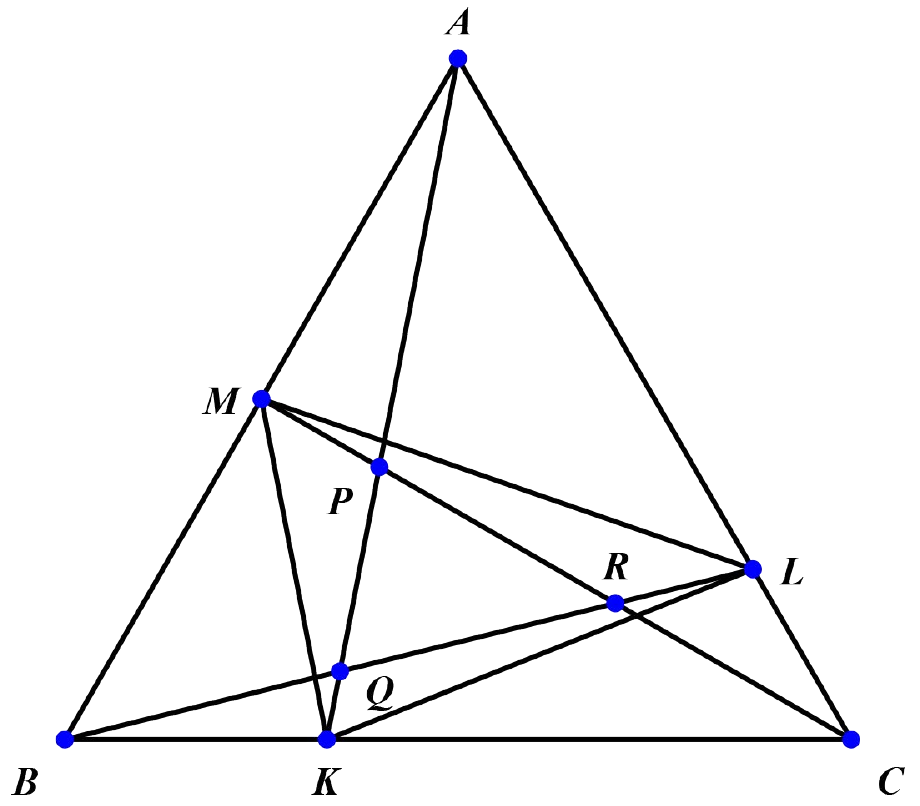}
\caption{Triangle}
\label{fig1}
\end{figure}

Then the area of the triangle $KLM$ is \[\frac{1+xyz}{(1+x)(1+y)(1+z)}\] and the area of the triangle 
$PQR$ is \[\frac{(1-xyz)^2}{(1+x+xy)(1+y+yz)(1+z+zx)}.\]
\end{theor}

As a consequence of the Routh's theorem we have the following Ceva's theorem.
\begin{theor}
The lines $AK$, $BL$ and $CM$ intersect at one point if and only if $xyz=1$.
\end{theor}

Our goal is to establish a three-dimensional analogue of Routh's theorem. Let us start with an arbitrary tetrahedron $ABCD$ of volume $1$. 
Choose a point $M$ on the line segment $AB$ and cut the tetrahedron $ABCD$ by the plane $CDM$ as displayed on Figure \ref{fig2}.
\begin{figure}[ht]\centering
\includegraphics[height=2in]{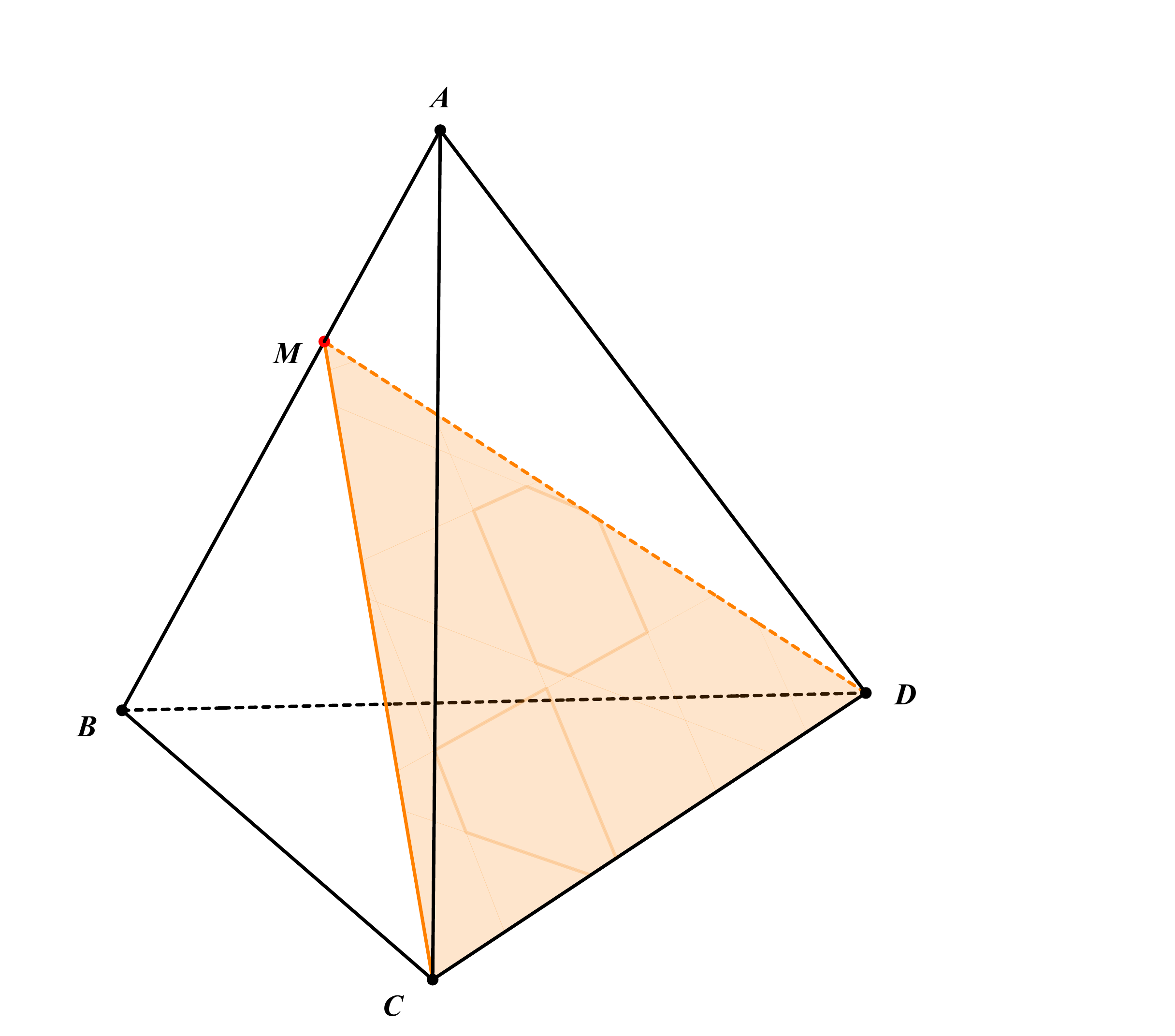}
\caption{Single cut}
\label{fig2}
\end{figure}

Next we construct three additional cutting planes: the plane containing the edge $AB$ and a point $K$ on the edge $CD$,
the plane containing the edge $BC$ and a point $L$ on the edge $DA$, and the plane 
containing  the edge $AD$ and a point $N$ on the edge $BC$. This way the points $M,N,K,L$ lie on the edges of the cycle $(ABCD)$.
All four cutting planes and their intersections are depicted on Figure \ref{fig3}.
\begin{figure}[ht]\centering
\includegraphics[height=2in]{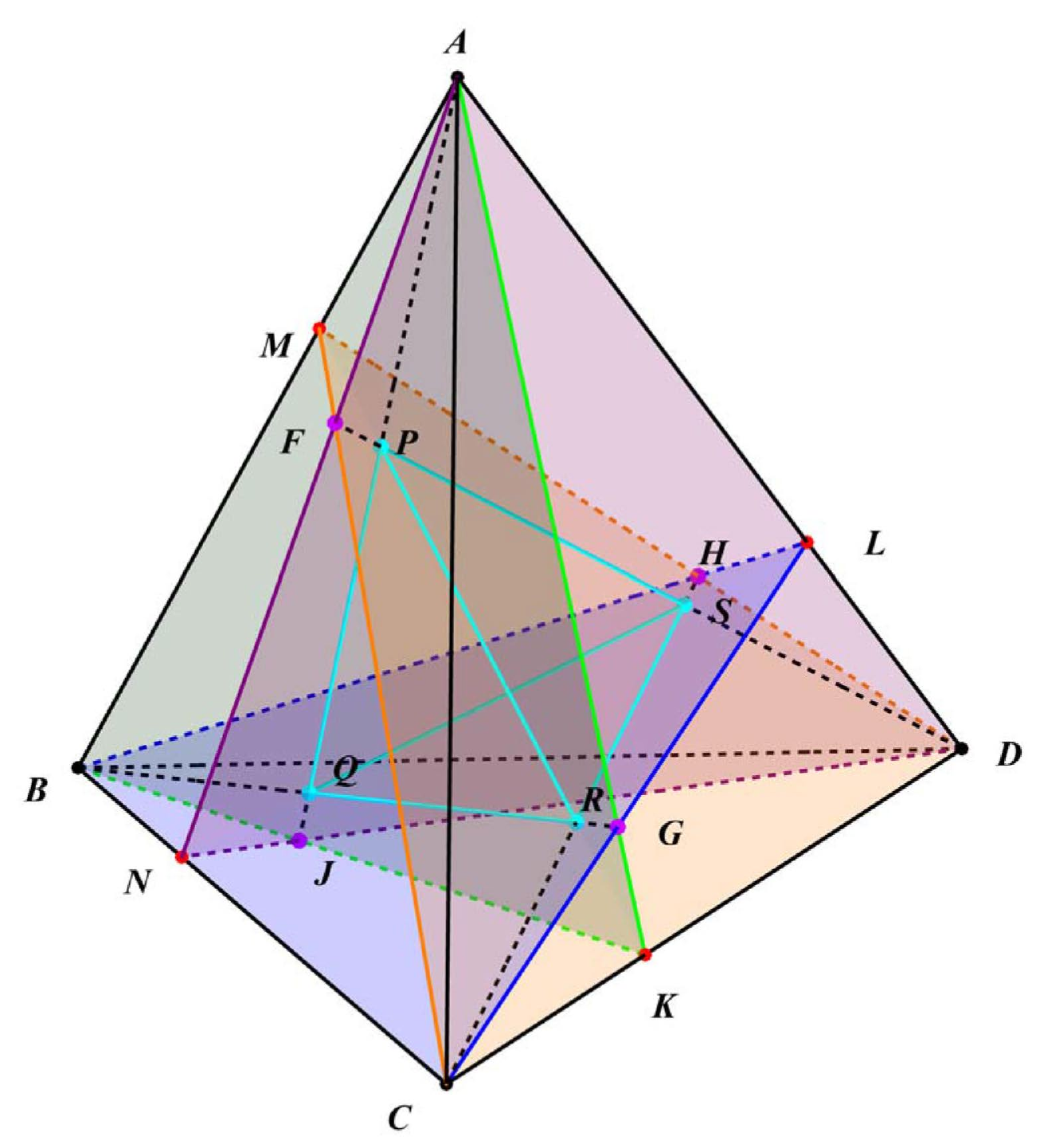}
\caption{Four cutting planes and their intersections }
\label{fig3}
\end{figure}

We will write $V_{ABCD}$ for the volume of a tetrahedron $ABCD$ and use analogous notation for volumes of other tetrahedra. 
Referring to Figure \ref{fig3}, here is our main result, Routh's theorem for tetrahedra:

\begin{theor}\label{tetra}
Let $ABCD$ be an arbitrary tetrahedron of volume $1$. Choose a point $M$ on the edge $AB$, a point $N$ on the edge $BC$, 
a point $K$ on the edge $CD$, and a point $L$ on the edge $DA$ such that $\frac{|CK|}{|KD|}=x$, $\frac{|DL|}{|LA|}=y$, 
$\frac{|AM|}{|MB|}=z$, and $\frac{|BN|}{|NC|}=t$. 
Then
\begin{equation}\label{e-2}
V_{KLMN}=\frac{|1-xyzt|}{(1+x)(1+y)(1+z)(1+t)}.
\end{equation}

The four planes given by the points $A, B, K$, points $B,C,L$, 
points $C,D,M$, and points $D,A,N$ enclose a tetrahedron $PQRS$ of the volume

\begin{equation}\label{e-1}
\begin{aligned}
&V_{PQRS} \\
&=\frac{|1-xyzt|^3}{(1+x+xy+xyz)(1+y+yz+yzt)(1+z+zt+ztx)(1+t+tx+txy)}.
\end{aligned}
\end{equation}
\end{theor}

\vskip 5pt

It is clear that Theorem \ref{tetra} implies the following Menelaus' theorem for tetrahedra.

\begin{theor}
The points $K,L,M,N$ in Theorem \ref{tetra} are coplanar if and only if $xyzt=1$.
\end{theor}

As another consequence of Theorem \ref{tetra} we obtain the analogue of Ceva's theorem.
\begin{theor}
The four planes given by points $A, B, K$, points $B,C,L$, points $C,D,M$ and points $D,A,N$ 
as in Theorem \ref{tetra} intersect at one point if and only if $xyzt=1$.
\end{theor}
 
The statement of Routh's theorem first appeared in the book \cite{r}, p.82.
Different proofs of Routh's theorem for triangles were given in papers \cite{a, bc, kl2, kv, n5}. 
It is rather suprising that, after an extensive search of the literature, we did not find any geometric 
proof generalizing this classical theorem from plane geometry that would cover tetrahedra, the three-dimensional analogues of triangles. 
The only references we have found are the papers \cite{k,yq}, where formulas generalizing 
Routh's theorem to higher dimensions are obtained by means of the vector analysis. 
The paper \cite{k} seems to be the first work showing formulas for the volumes of both tetrahedra in Theorem \ref{tetra}. 
The notation of \cite{k} follows that of the paper \cite{bn} and an earlier work \cite{n1},
and the reader needs to adjust it to the notation used in the present article. 
It is rather unfortunate that paper \cite{k} is written in Slovak and paper 
\cite{yq} in Chinese language which makes them inaccessible for most readers.
In fact, we did not know about these papers before we obtained main results of this article.

We intent to give a more thorough review of the literature related to generalizations of theorems of Routh, 
Ceva, and Menelaus in our forthcoming paper that starts with the geometry of tetrahedra and leads 
to an algebraic formula related to cycles of length $n$ that is interesting on its own even without the relationship to the Routh's theorem.

It was already observed in \cite{w4} that if the points $K,L,M,N$ are coplanar, then $xyzt=1$. This is the only case when $V_{KLMN}=0$. 
By examining Figure \ref{fig3}, one can notice that the vertices $P$ and $R$ of the tetrahedron $PQRS$ lie on the line $KM$
while the vertices $S$ and $Q$ lie on the line $LN$. Since this implies that $V_{KLMN}=0$ if and only if $V_{PQRS}=0$, we also conclude that $V_{PQRS}=0$ only when
$xyzt=1$. Note also that when we fix, for example, $y,z$ and $t$ and let $x$ increase starting at $0$, the product $xyzt$ also
increases starting at $0$, so the value $xyzt=1$ occurs when the point $K$ moves from $C$ to $D$. 
Thus, before this occurrence, the points $R$ and $C$ belong to the same halfspace determined by the plane passing through the points $P, Q$ and $S$.
This observation suggests that Figure \ref{fig3} corresponds to the case where $xyzt<1$.
In what follows we assume that, in Theorem \ref{tetra}, $xyzt<1$ and derive formulas for the volumes of tetrahedra $KLMN$ and $PQRS$. 
After that is done, we extend the formulas to the case $xyzt>1$ by changing the orientation 
and applying a suitable substitution of the parameters $x,y,z,t$. 

\section{Proof of (\ref{e-2})}
Assume first that $xyzt<1$ and refer to Figure \ref{fig4}.

\begin{figure}[ht]\centering
\includegraphics[height=3.5in]{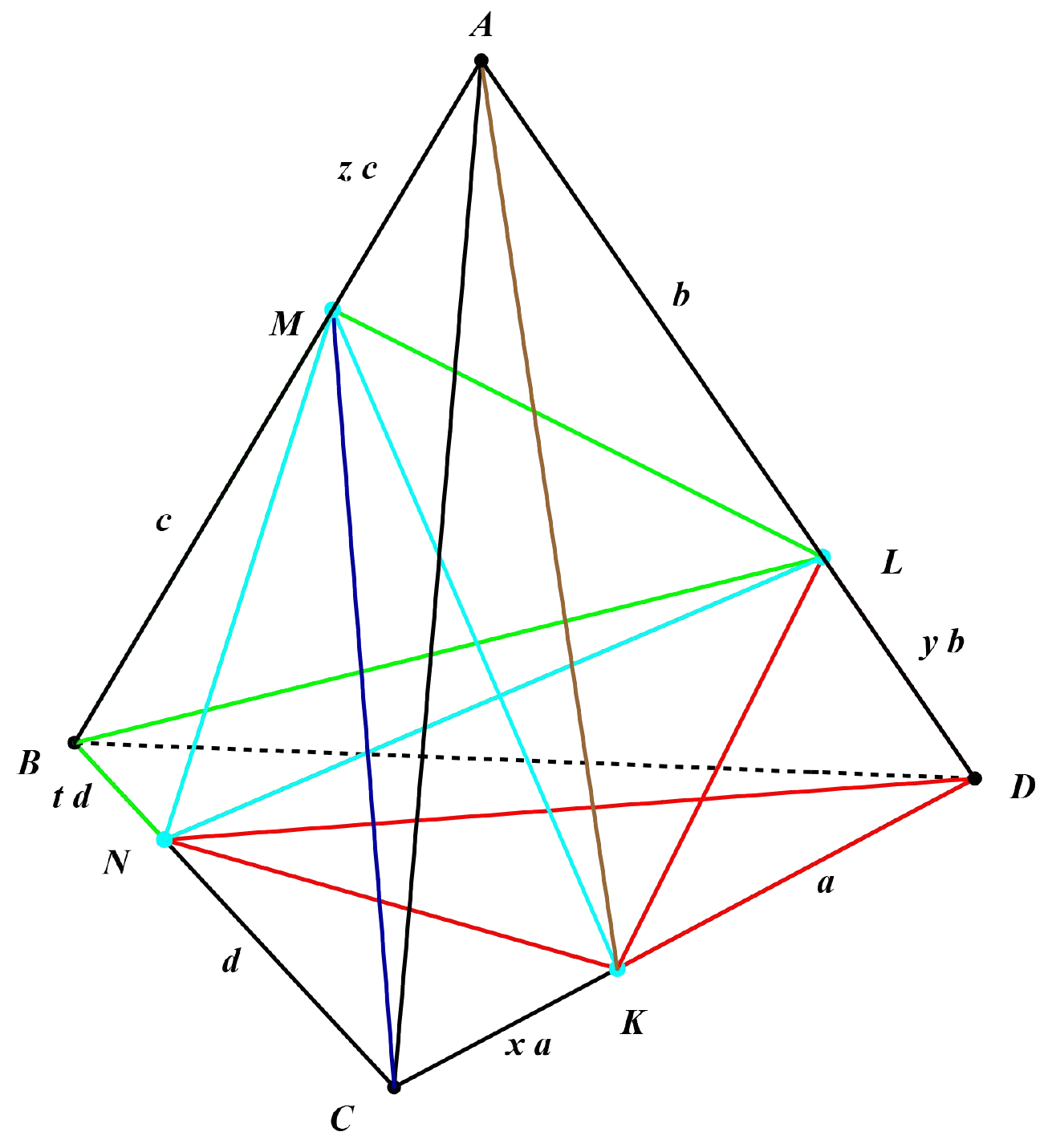}
\caption{Notation $(xyzt<1)$}
\label{fig4}
\end{figure}

The following lemma will be used repetitively, often without explicit reference.

\begin{lm}\label{lm1}
In the notation of Figure \ref{fig2}, we have
$$
V_{AMCD}=V_{ABCD}\frac{|AM|}{|AB|}.
$$
\end{lm}

We will also utilize the following consequence of Lemma \ref{lm1} on a number of occasions.

\begin{lm}\label{l15} In the notation of Figure \ref{fig5}, 
\begin{figure}[ht]\centering
\includegraphics[height=2in]{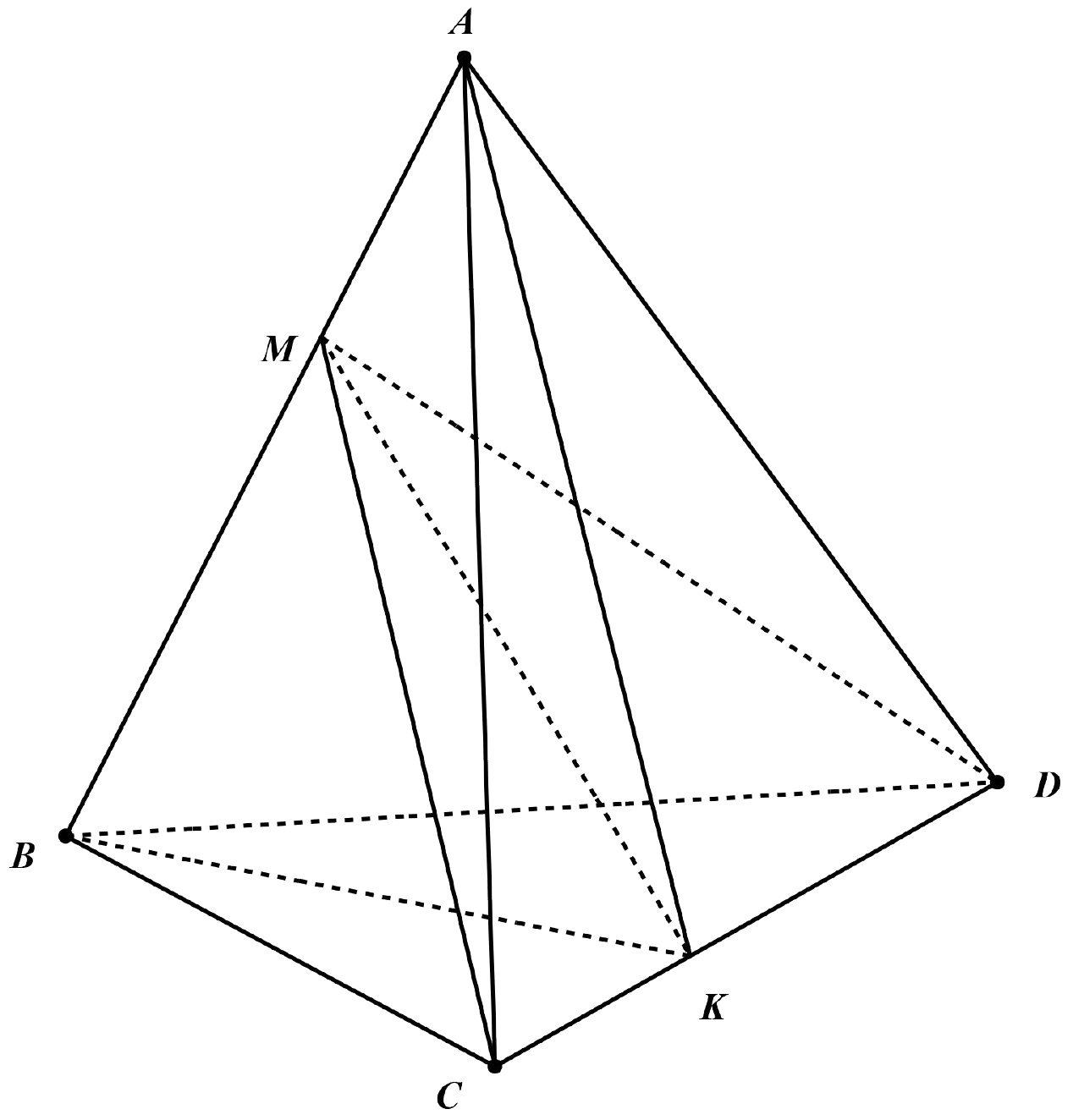}
\caption{Tetrahedra ACKM and ADKM}
\label{fig5}
\end{figure}
we have 
\[V_{AKCM}=V_{ABCD}\frac {|AM|}{|AB|}\frac{|CK|}{|CD|} \text{ \ and \ } 
V_{AMKD}=V_{ABCD}\frac{|AM|}{|AB|}\frac{|DK|}{|DC|}. \]
\end{lm}

Referring to Figure \ref{fig4}, the tetrahedron $ABCD$ contains seven smaller
tetrahedra $KLMN$, $AKLM$, $BLMN$, $CKMN$, $DKLN$, $ACKM$, and $BDLN$. 

The critical observation that
\begin{equation}\label{e0} 
V_{ABCD}=V_{KLMN}+V_{AKLM}+V_{BLMN}+V_{CKMN}+V_{DKLN}+V_{ACKM}+V_{BDLN}
\end{equation}
can be seen from Figures \ref{fig6} and \ref{fig7}.

\begin{figure}[ht]
\centering
\begin{minipage}[b]{0.48\linewidth}
\includegraphics[height=2.5in]{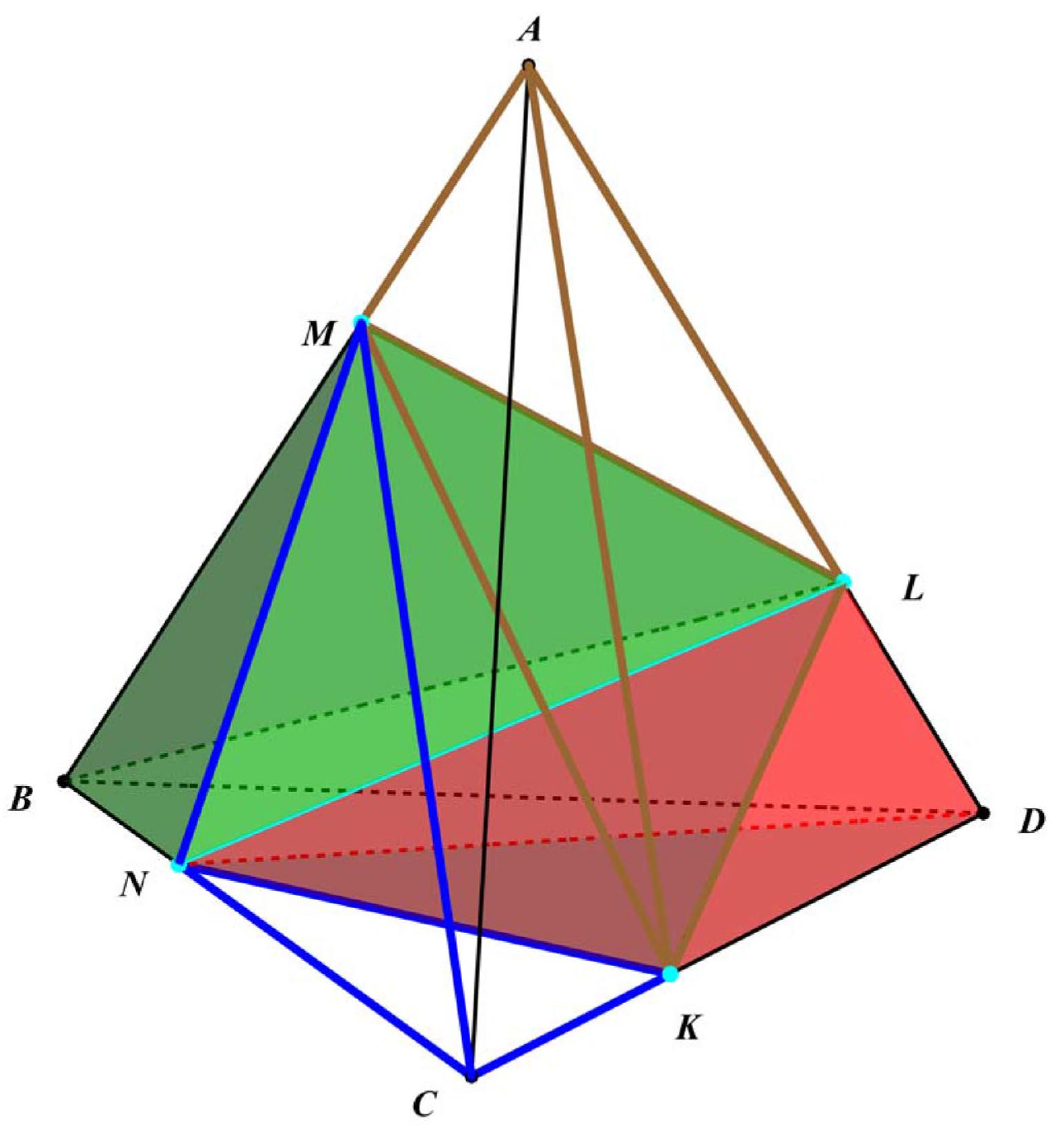}
\caption{Tetrahedra BLMN and DKLN}
\label{fig6}
\end{minipage}
\quad
\begin{minipage}[b]{0.48\linewidth}
\includegraphics[height=2.5in]{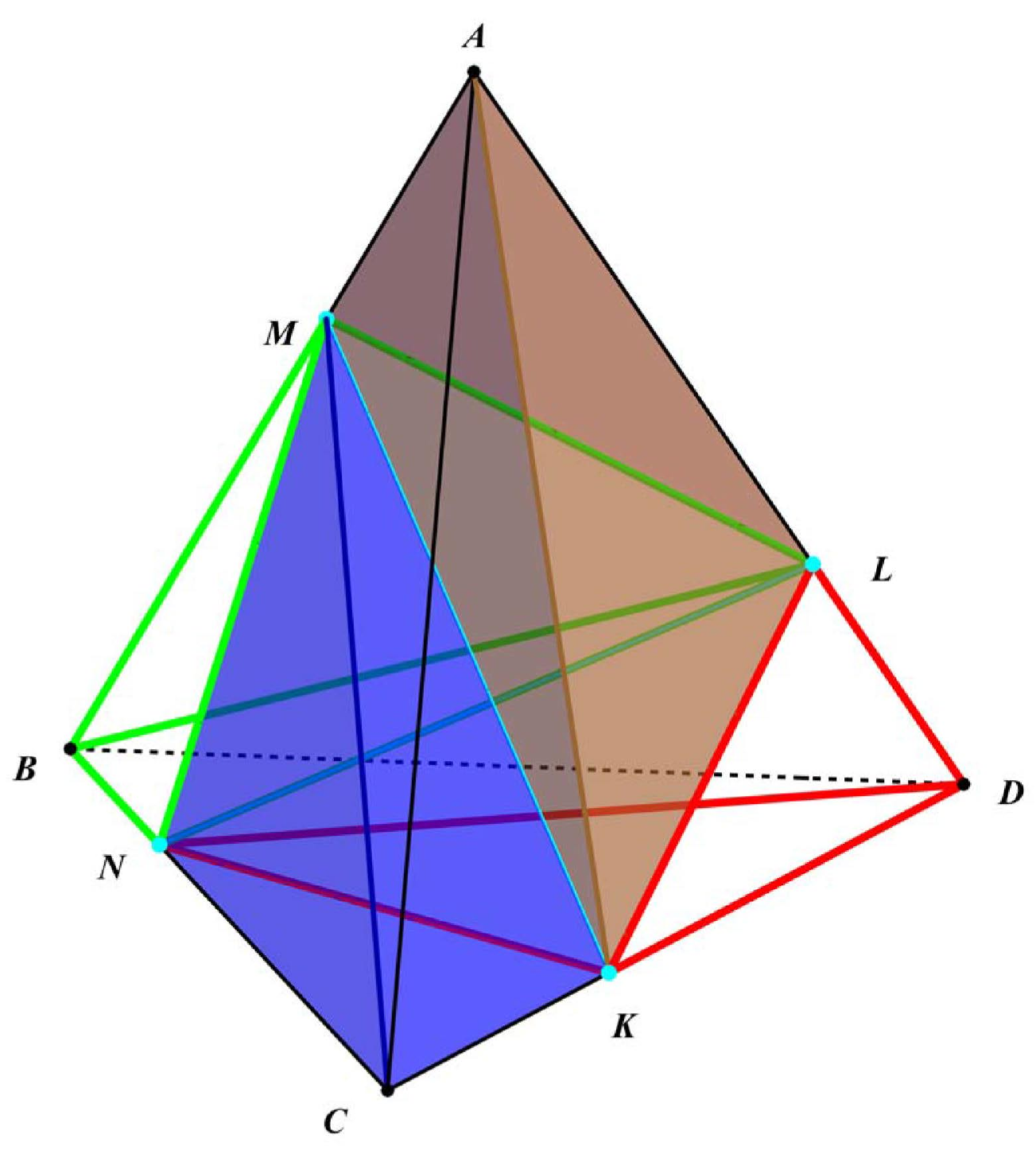}
\caption{Tetrahedra AKLM and CKMN}
\label{fig7}
\end{minipage}
\end{figure}


By Lemma \ref{l15}, we have
$$
V_{ACKM}=V_{ABCD}\frac{|AM|}{|AB|}\frac{|CK|}{|CD|}=\frac z{1+z}\frac x{1+x},
$$
$$
V_{BDLN}=V_{ABCD}\frac{|BN|}{|BC|}\frac{|DL|}{|DA|}=\frac t{1+t} \frac y{1+y}.
$$
Applying Lemma \ref{l15} in the second step of each of the following four calculations, we obtain
$$
V_{AKLM}=V_{AKDM}\frac{|AL|}{|AD|}=V_{ABCD}\frac{|AM|}{|AB|}\frac{|DK|}{|DC|}\frac{|AL|}{|AD|}=\frac z{1+z}\frac 1{1+x}\frac 1{1+y},
$$
$$
V_{BLMN}=V_{BALN}\frac{|BM|}{|BA|}=V_{ABCD}\frac{|AL|}{|AD|}\frac{|BN|}{|BC|}\frac{|BM|}{|BA|}=\frac 1{1+y}\frac t{1+t}\frac 1{1+z},
$$
$$
V_{CKMN}=V_{CKMB}\frac{|CN|}{|CB|}=V_{ABCD}\frac{|BM|}{|BA|}\frac{|CK|}{|CD|}\frac{|CN|}{|CB|}=\frac 1{1+z}\frac x{1+x}\frac 1{1+t},
$$
$$
V_{DKLN}=V_{DCLN}\frac{|DK|}{|DC|}=V_{ABCD}\frac{|CN|}{|CB|}\frac{|DL|}{|DA|}\frac{|DK|}{|DC|}=\frac 1{1+t}\frac y{1+y}\frac 1{1+x}.
$$

Therefore equation (\ref{e0}) yields
\[\begin{aligned}V_{KLMN}&=1-\frac x{(1+x)(1+z)(1+t)}-\frac y{(1+x)(1+y)(1+t)}-\frac z{(1+x)(1+y)(1+z)}\\
&-\frac t{(1+y)(1+z)(1+t)}-\frac {xz}{(1+x)(1+z)}-\frac {yt}{(1+y)(1+t)}\\
&= \frac{1-xyzt}{(1+x)(1+y)(1+z)(1+t)}.\end{aligned}\]

\vskip 5pt
Assume now that $xyzt>1$. In order to reduce this case to the case $xyzt<1$, we can put the 
tetrahedron $ABCD$ into a different perspective, namely, let us place the vertex $D$ at the original position of the vertex $A$
and run the edge $DB$ along the original edge $AC$, in the same direction. If we now consider the cycle $(DCBA)$ instead of the 
original cycle $(ABCD)$, the parameters $x,y,z,t$ switch to $\frac 1z,\frac 1y, \frac 1x, \frac 1t$, respectively. Since  
$\frac 1z\frac 1y\frac 1x\frac 1t<1$, we can apply the previous formula which yields
$$
V_{KLMN}=\frac{xyzt-1}{(1+x)(1+y)(1+z)(1+t)}.
$$
\qed

\section{Proof of (\ref{e-1})}

Now we proceed to compute the volume of the smaller tetrahedron $PQRS$, 
carved up inside the original tetrahedron $ABCD$ by the four cutting planes described earlier.
Assuming again $xyzt<1$ and using the notation of Figure \ref{fig8}, 
\begin{figure}[ht]\centering
\includegraphics[height=3.5in]{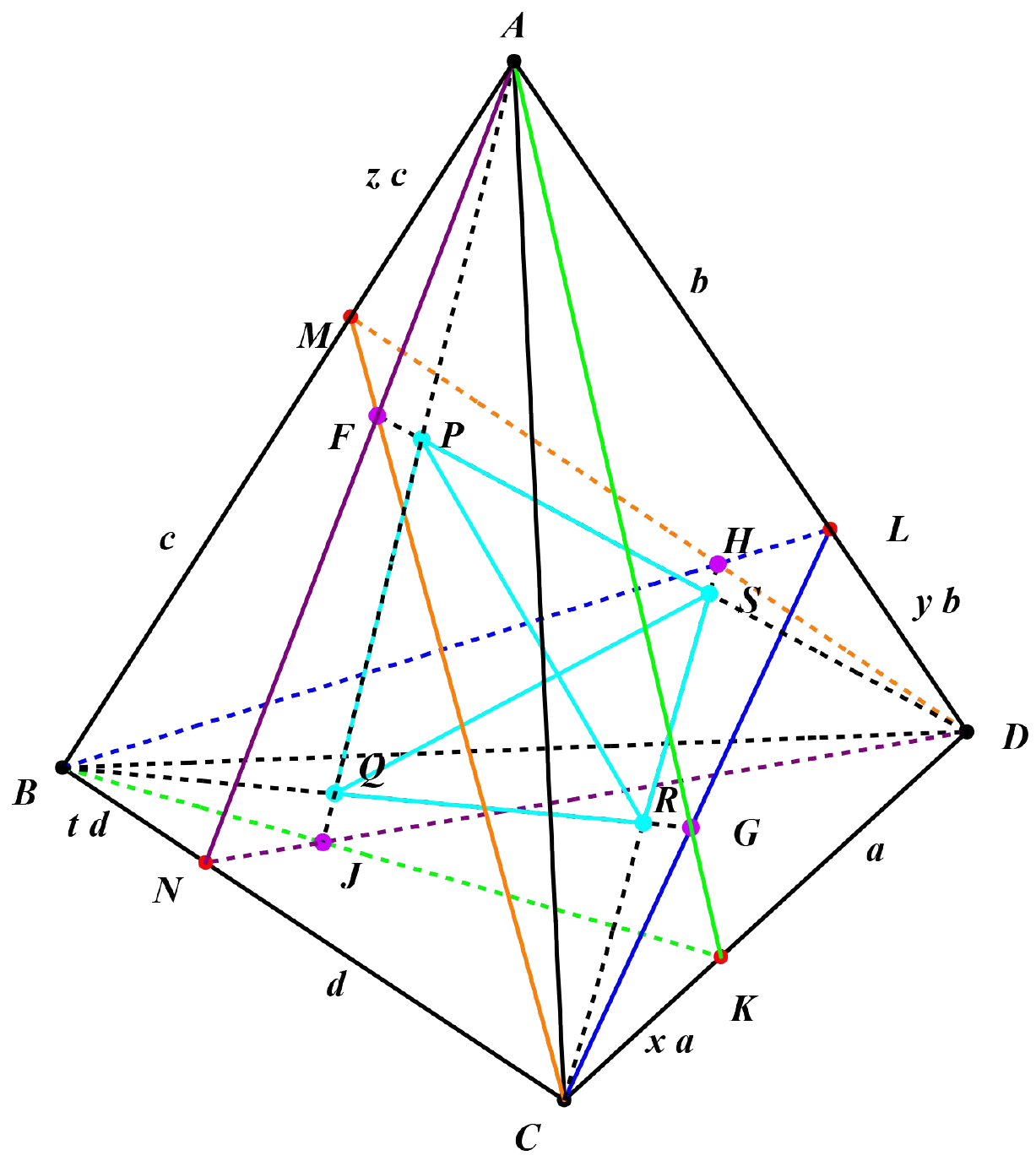}
\caption{Notation $(xyzt<1)$}
\label{fig8}
\end{figure}
we will consider seven smaller tetrahedra $APKD$, $BQLA$, $CRMB$, $DSNC$, $AMCK$, $BNDL$, and $PQRS$.

The critical observation that 
\begin{equation}\label{e1} V_{ABCD}=V_{APKD}+V_{BQLA}+V_{CRMB}+V_{DSNC}+V_{AMCK}+V_{BNDL}+V_{PQRS}\end{equation}
can be seen from Figures \ref{fig9} and \ref{fig10}.

\begin{figure}[ht]
\centering
\begin{minipage}[b]{0.4743\linewidth}
\includegraphics[height=2.5in]{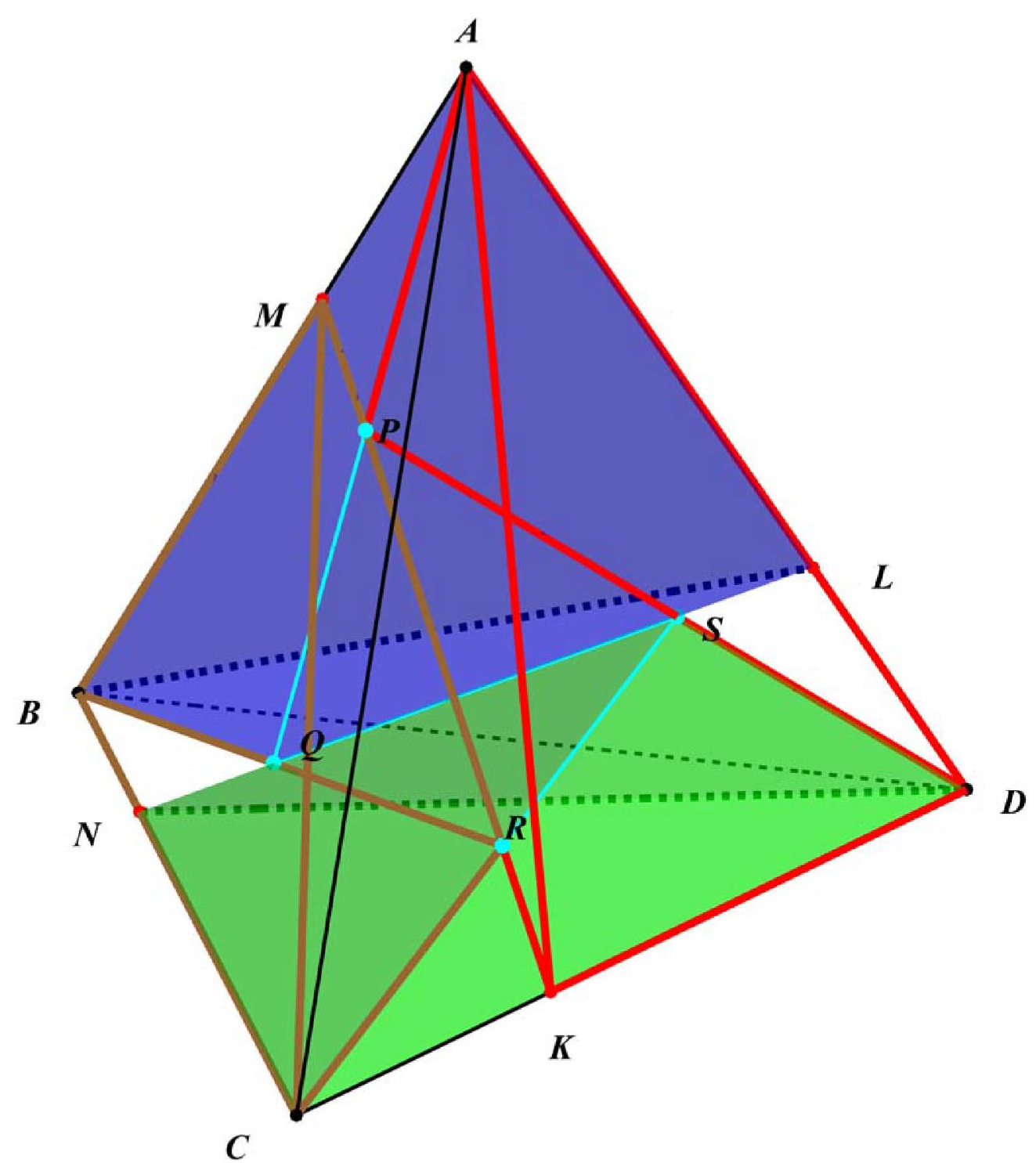}
\caption{Tetrahedra BQLA and DSNC}
\label{fig9}
\end{minipage}
\quad
\begin{minipage}[b]{0.4897\linewidth}
\includegraphics[height=2.5in]{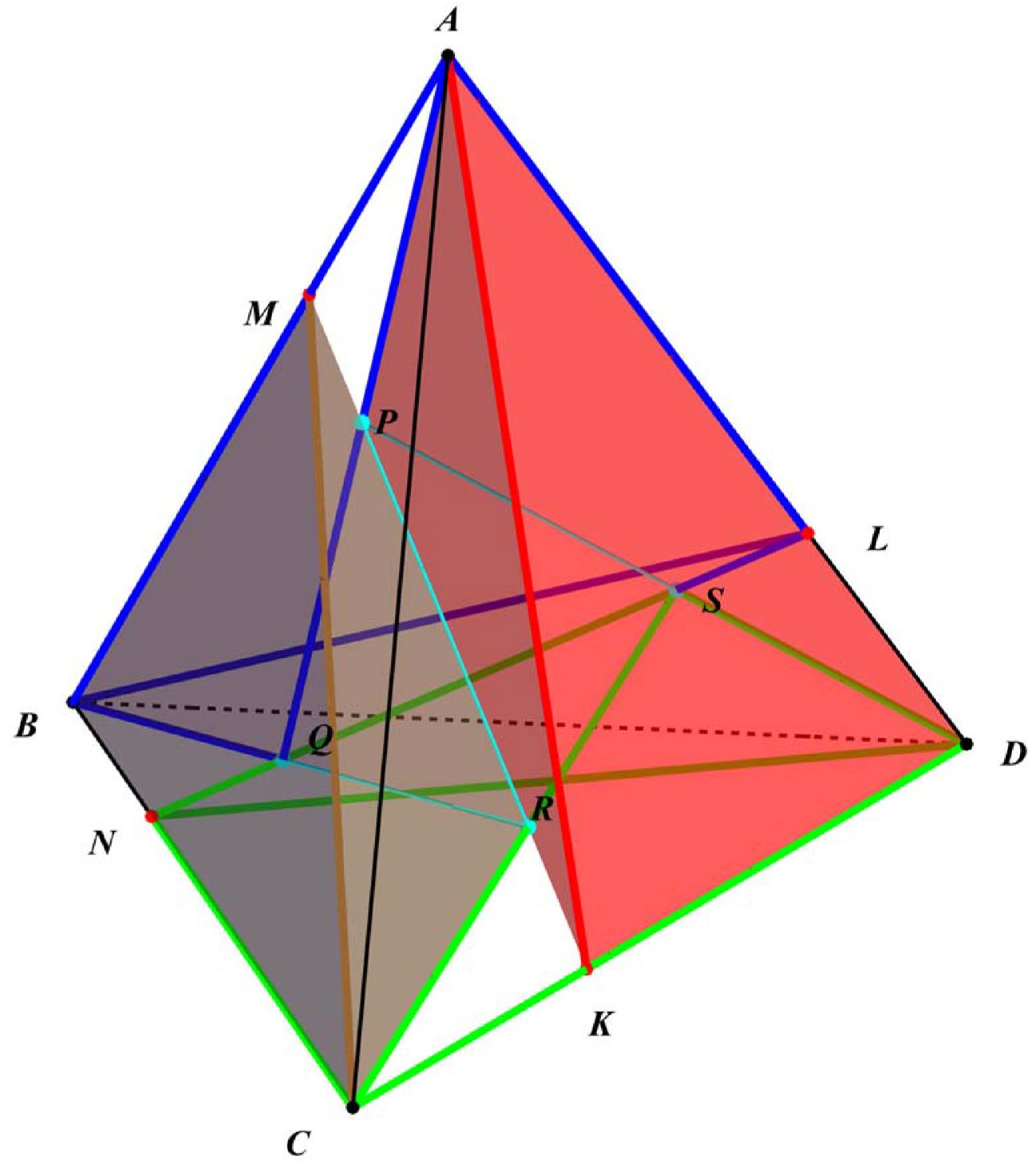}
\caption{Tetrahedra APKD and CRMB}
\label{fig10}
\end{minipage}
\end{figure}


Using Lemma \ref{l15}, we compute 

\[V_{AMCK}=V_{ABCD}\frac{|AM|}{|AB|}\frac{|CK|}{|CD|}=
\frac{z}{1+z}\frac{x}{1+x}\] and 
\[V_{BNDL}=V_{ABCD}\frac{|DL|}{|DA|}\frac{|BN|}{|BC|}=\frac{y}{1+y}\frac{t}{1+t}.\]

\vskip 5pt
Let us now calculate the volumes of the four remaining tetrahedra in the right-hand side of (\ref{e1}).

\begin{lm}\label{lm3}
Consider the triangle $ABC$ as in Figure \ref{fig11}. If $\frac{|AM|}{|MB|}=v$ and $\frac{|BK|}{|KC|}=u$, then
\begin{figure}[ht]\centering
\includegraphics[height=1.5in]{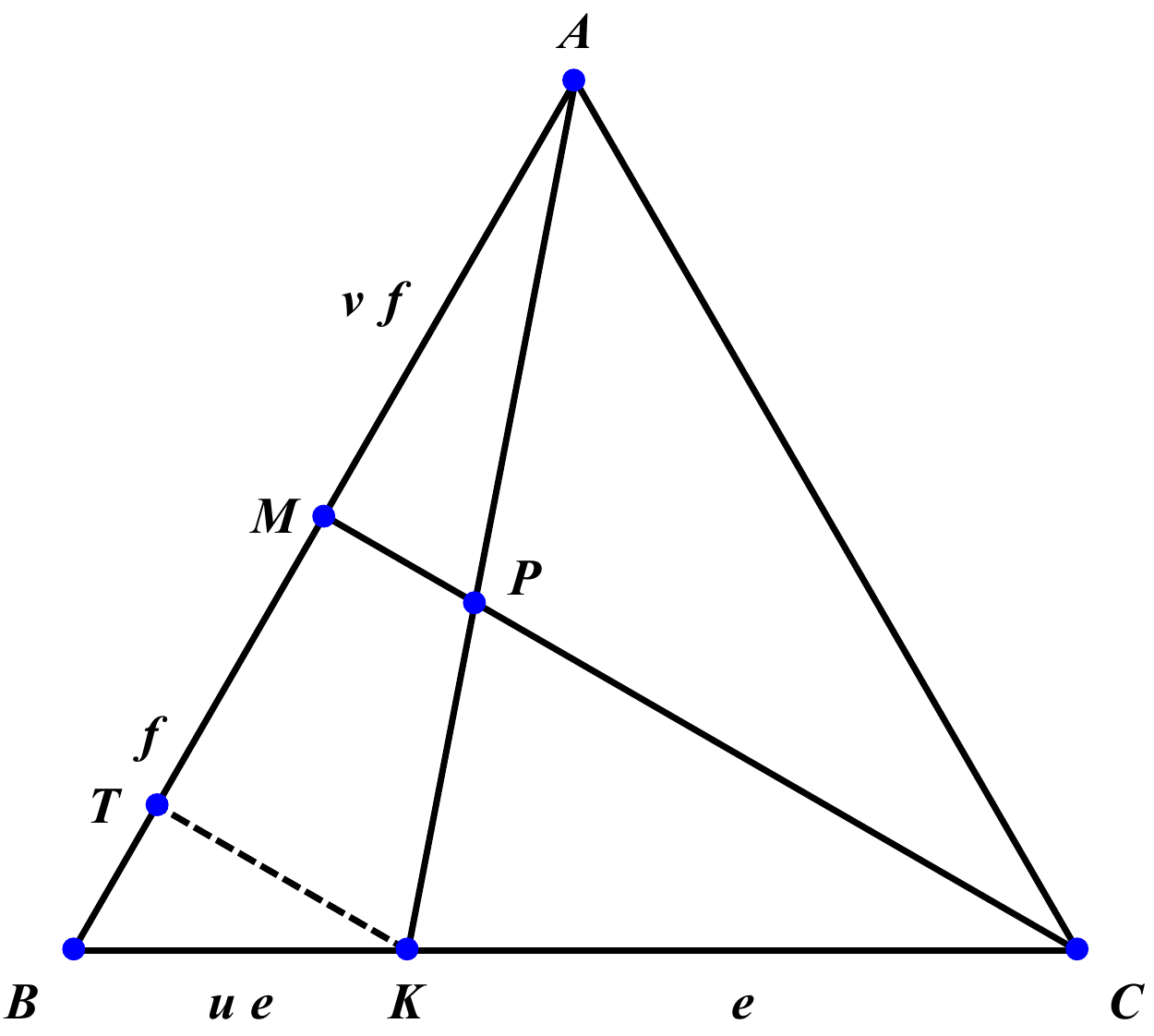}
\caption{Ratio $\frac{|AP|}{|PK|}$}
\label{fig11}
\end{figure}
\[\frac{|AP|}{|PK|}=v(1+u).\]
\end{lm}
\begin{proof}
If $KT \parallel CM$, then $|MT|=\frac{f}{1+u}$. Therefore 
\[\frac{|AP|}{|PK|}=\frac{|AM|}{|MT|}=v(1+u).\]
\end{proof} 

Applying Lemma \ref{lm1} several times, we obtain
\[\begin{aligned}V_{APKD}=&V_{APCD}\frac{|DK|}{|DC|}=V_{AFCD}\frac{|DP|}{|DF|}\frac{|DK|}{|DC|}=
V_{AMCD}\frac{|CF|}{|CM|}\frac{|DP|}{|DF|}\frac{|DK|}{|DC|}\\
=&V_{ABCD}\frac{|AM|}{|AB|}\frac{|CF|}{|CM|}\frac{|DP|}{|DF|}\frac{|DK|}{|DC|}.
\end{aligned}\]
It is immediate that the ratios 
\[\frac{|AM|}{|AB|}=\frac{z}{1+z} \text{ \ and \ } \frac{|DK|}{|DC|}=\frac{1}{1+x}.\]

\vskip 5pt
Lemma \ref{lm3} applied to the triangle $ABC$ gives
\[\frac{|CF|}{|FM|}=\frac1t\left (1+\frac1z \right)=\frac{1+z}{zt}.\] 
Another application of Lemma \ref{lm3}, now to triangle $CDM$, yields
\[\frac{|DP|}{|PF|}=\frac1x\left (1+\frac{1+z}{zt}\right )=\frac{1+z+zt}{ztx}.\]
Using these we derive
\[\frac{|CF|}{|CM|}=\frac{|CF|}{|CF|+|FM|}=\frac{\frac{|CF|}{|FM|}}{\frac{|CF|}{|FM|}+1}=
\frac{\frac{1+z}{zt}}{1+\frac{1+z}{zt}}=\frac{1+z}{1+z+zt}\] and 
\[\frac{|DP|}{|DF|}=\frac{|DP|}{|DP|+|PF|}=\frac{\frac{|DP|}{|PF|}}{\frac{|DP|}{|PF|}+1}=
\frac{\frac{1+z+zt}{ztx}}{1+\frac{1+z+zt}{ztx}}=\frac{1+z+zt}{1+z+zt+ztx}.\]
Therefore
\[V_{APKD}=\frac{z}{1+z}\frac{1+z}{1+t+zt}\frac{1+z+zt}{1+z+zt+ztx}=\frac{z}{(1+x)(1+z+zt+ztx)}.\]
Similarly, 
\[\begin{aligned}V_{BQLA}=&V_{BQDA}\frac{|AL|}{|AD|}=V_{BJDA}\frac{|AQ|}{|AJ|}\frac{|AL|}{|AD|}
=V_{BNDA}\frac{|DJ|}{|DN|}\frac{|AQ|}{|AJ|}\frac{|AL|}{|AD|}\\
=&V_{ABCD}\frac{|BN|}{|BC|}\frac{|DJ|}{|DN|}\frac{|AQ|}{|AJ|}\frac{|AL|}{|AD|}
\end{aligned}\]
and 
\[\frac{|BN|}{|BC|}=\frac{t}{1+t}, \quad \frac{|AL|}{|AD|}=\frac{1}{1+y}.\]

\vskip 5pt
Lemma \ref{lm3} applied to the triangle $BCD$ gives
\[\frac{|DJ|}{|JN|}=\frac1x\left (1+\frac1t \right )=\frac{1+t}{xt},\] hence \[\frac{|DJ|}{|DN|}=\frac{1+t}{1+t+xt}.\]
Lemma \ref{lm3} applied to the triangle $AND$ yields \[\frac{|AQ|}{|QJ|}
=\frac1y\left (1+\frac{1+t}{xt}\right )=\frac{1+t+xt}{xyt}\] and therefore \[\frac{|AQ|}{|AJ|}=\frac{1+t+xt}{1+t+xt+xyt}.\]
Thus 
\[V_{BQLA}=\frac{t}{1+t}\frac{1+t}{1+t+xt}\frac{1+t+xt}{1+t+xt+xyt}=\frac{t}{(1+y)(1+t+tx+txy)}.\]
Proceeding as before, we have
\[\begin{aligned}V_{CRMB}=&V_{CRAB}\frac{|BM|}{|BA|}=V_{CGAB}\frac{|BR|}{|BG|}\frac{|BM|}{|BA|}
=V_{CLAB}\frac{|CG|}{|CL|}\frac{|BR|}{|BG|}\frac{|BM|}{|BA|}\\
=&V_{ABCD}\frac{|AL|}{|AD|}\frac{|CG|}{|CL|}\frac{|BR|}{|BG|}\frac{|BM|}{|BA|},
\end{aligned}\]
where 
\[\frac{|AL|}{|AD|}=\frac{1}{1+y}, \quad \frac{|BM|}{|BA|}=\frac{1}{1+z}.\]

\vskip 5pt
Lemma \ref{lm3} applied to the triangle $ACD$ implies
\[\frac{|CG|}{|GL|}=x(1+y),\] and this gives \[\frac{|CG|}{|CL|}=\frac{x(1+y)}{1+x+xy}.\]
Also, Lemma \ref{lm3} applied to the triangle $ABD$ gives 
\[\frac{|BH|}{|HL|}=\frac1z\left (1+\frac1y \right )=\frac{1+y}{yz}.\] 
Now, applied to the triangle $ABD$, Lemma \ref{lm3} yields  
\[\frac{|BR|}{|RG|}=\frac{1+y}{yz}\left (1+\frac{1}{x(1+y)}\right )=\frac{1+x+xy}{xyz},\] and we derive
\[\frac{|BR|}{|BG|}=\frac{1+x+xy}{1+x+xy+xyz}.\]
Therefore
\[V_{CRMB}=\frac{1}{1+y}\frac{x(1+y)}{1+x+xy}\frac{1+x+xy}{1+x+xy+xyz}=\frac{x}{(1+z)(1+x+xy+xyz)}.\]
Finally, 
\[\begin{aligned}V_{DSNC}=&V_{DSBC}\frac{|NC|}{|BC|}=V_{DHBC}\frac{|CS|}{|CH|}\frac{|NC|}{|BC|}=V_{DLBC}\frac{|BH|}{|BL|}\frac{|CS|}{|CH|}\frac{|NC|}{|BC|}\\
=&V_{ABCD}\frac{|DL|}{|DA|}\frac{|BH|}{|BL|}\frac{|CS|}{|CH|}\frac{|NC|}{|BC|}
\end{aligned}\]
and 
\[\frac{|DL|}{|DA|}=\frac{y}{1+y}, \quad \frac{|NC|}{|BC|}=\frac{1}{1+t}.\]

\vskip 5pt
It follows from Lemma \ref{lm3} for the triangle $ABD$ that 
\[\frac{|BH|}{|HL|}=\frac1z\left (1+\frac1y \right )=\frac{1+y}{yz},\] so
\[\frac{|BH|}{|BL|}=\frac{1+y}{1+y+yz}.\]
An application of Lemma \ref{lm3} to the triangle $BCL$ results in
\[\frac{|CS|}{|SH|}=\frac1t\left (1+\frac{1+y}{yz}\right )=\frac{1+y+yz}{yzt}\] and 
\[\frac{|CS|}{|CH|}=\frac{1+y+yz}{1+y+yz+yzt},\] and we obtain
\[V_{DSNC}=\frac{y}{1+y}\frac{1+y}{1+y+yz}\frac{1+y+yz}{1+y+yz+yzt}=\frac{y}{(1+t)(1+y+yz+yzt)}.\]

Now equation (\ref{e1}) yields
\[\begin{aligned}V_{PQRS}=&1-\frac{xz}{(1+x)(1+z)}-\frac{yt}{(1+y)(1+t)}-\frac{x}{(1+z)(1+x+xy+xyz)} \\
&-\frac{y}{(1+t)(1+y+yz+yzt)} -\frac{z}{(1+x)(1+z+zt+ztx)}\\&-\frac{t}{(1+y)(1+t+tx+txy)}.
\end{aligned}\]
It can be verified directly or with the help of computer that the last expression equals
\[\frac{(1-xyzt)^3}{(1+x+xy+xyz)(1+y+yz+yzt)(1+z+zt+ztx)(1+t+tx+txy)}.\]

\vskip 5pt
If $xyzt>1$, then we can perform the same substitution of the parameters  $x,y,z,t$ as in the previous section to see that

\[V_{PQRS}= \frac{(xyzt-1)^3}{(1+x+xy+xyz)(1+y+yz+yzt)(1+z+zt+ztx)(1+t+tx+txy)}.\]
\qed

{\bf Acknowledgement.} The authors are endebted to Professor Jose Alfredo Jimenez for his encouragement and help with the images which the present article heavily relies on.


\begin{thebibliography}{99}


\bibitem{a} Ayoub, Ayoub B., {\it Routh's theorem revisited}, Mathematical Spectrum 44 (1) (2011/2012), 24--27. 

\bibitem{bc} B\'{e}nyi, \'{A}rp\'{a}d; \'{C}urgus, Branko, {\it A generalization of Routh's triangle theorem}. (English summary) 
Amer. Math. Monthly  120  (2013),  no. 9, 841--846. 




\bibitem{bn} Budinsk\'{y}, Bruno; N\'{a}den\'{\i}k, Zbyn\v{e}k, {\it Mehrdimensionales Analogon zu den S\"{a}tzen von Menelaos und Ceva}. (German. Czech summary) 
\v{C}asopis P\v{e}st. Mat.  97  (1972), 75--77, 95. 













\bibitem{k} Klein, Tomas, {\it A certain generalization of the theorems of Menelaos and Ceva}. (Slovak. German summary) 
\v{C}asopis P\v{e}st. Mat. 98  (1973), 22--25. 


\bibitem{kl2} Klamkin, Murray S.; Liu, A., {\it Three more proofs of Routh's theorem}, Crux Mathematicorum 7 (1981) 199--203. 


\bibitem{kv} Kline, J.S. ; Velleman, D. (1995) {\it Yet another proof of Routh's theorem}, (1995) Crux Mathematicorum 21 (1995), 37--40.  









\bibitem{n1} N\'{a}den\'{\i}k, Zbyn\v{e}k, {\it L'\'{e}largissement du th\'{e}or\'{e}me de M\'{e}n\'{e}la\"{u}s et de C\'{e}va sur les figures $n$-dimensionnelles}. (Czech. Russian, French summary) \v{C}asopis P\v{e}st. Mat.   81  (1956), 1--25. 




\bibitem{n5} Niven, Ivan, {\it A new proof of Routh's theorem}. Math. Mag.  49  (1976), no. 1, 25--27. 

\bibitem{r} B. J. Routh, {\it A Treatise on Analytical Statics with Numerous Examples}, Vol. 1, second edition. Cambridge University Press, London, 1909, 
available at http://www.archive.org/details/texts.




\bibitem{w4} Wernicke, Paul, {\it The Theorems of Ceva and Menelaus and Their Extension}. Amer. Math. Monthly  34  (1927),  no. 9, 468--472. 

\bibitem{yq} Yang, Shi Guo; Qi, Ji Bing, {\it Higher-dimensional Routh theorem}. (Chinese. English, Chinese summary) J. Math. (Wuhan)  31  (2011),  no. 1, 152--156. 


\end{thebibliography}
\end{document}